\documentclass[a4paper,12pt]{amsart}
\usepackage[abbrev]{amsrefs}
\usepackage{amscd}
\usepackage{amssymb}
\usepackage{mathtools}

\theoremstyle{plain}
  \newtheorem{thm}{Theorem}[section]
  \newtheorem{lem}[thm]{Lemma}
  
  \newtheorem{cor}[thm]{Corollary}
  \newtheorem{prop}[thm]{Proposition}
  \newtheorem{claim}[thm]{Claim}

\theoremstyle{definition}
  \newtheorem{dfn}[thm]{Definition}
  
  \newtheorem{ex}[thm]{Example}

\theoremstyle{remark}
  \newtheorem{rmk}[thm]{Remark}

  \newtheorem*{ack}{Acknowledgment}

\numberwithin{equation}{section}

\makeindex

\DeclareMathOperator{\supp}{supp}

\DeclareMathOperator{\esssup}{ess\sup} 
\DeclareMathOperator{\essinf}{ess\inf}

\newcommand{\dis}{\operatorname{dis}}

\newcommand{\diam}{\operatorname{diam}}

\newcommand{\ldimH}{\underline{\mathrm{dim}}_{\mathrm{H}}}

\newcommand{\udimH}{\overline{\mathrm{dim}}_{\mathrm{H}}} 
\newcommand{\essldimH}{\underline{\mathrm{dim}}_{\mathrm{H}}^{\mathrm{ess}}} 
\newcommand{\essudimH}{\overline{\mathrm{dim}}_{\mathrm{H}}^{\mathrm{ess}}}

\usepackage[%
 setpagesize=false,%
 bookmarks=true,%
 bookmarksdepth=tocdepth,%
 bookmarksnumbered=true,
 colorlinks=true,%
 pdftitle={},%
 pdfsubject={},%
 pdfauthor={},%
 pdfkeywords={}%
]{hyperref}

\begin{document}

\title
{The semicontinuity of dimensions of measures}

\begin{abstract}
  In this paper, we consider certain uniformity conditions on the radii of balls.
  Under these conditions, we study the semicontinuity of the Hausdorff dimension of measures with respect to measured Gromov--Hausdorff convergence and box convergence.  
  We also examine the semicontinuity of other dimensions of measures, which are similar to the Hausdorff dimension of measures. 

\end{abstract}

\author{Daiki Takagi}
\address{Mathematical Institute, Tohoku University, Sendai 980-8578, Japan}
\email{takagi.daiki.t1@dc.tohoku.ac.jp}

\date{\today}


\maketitle

\setcounter{tocdepth}{3}

\tableofcontents

\section{Introduction}
\label{sec:intro}

We are interested in notions of dimension for metric measure spaces, especially the relation between the dimensions of a convergent sequence of metric measure spaces and the dimension of the limit. 
A typical example of a dimension for metric measure spaces is the Hausdorff dimension of measures. 
It is defined via the local dimension, which is the growth rate of the measure of balls centered at each point. 
A definition of it can be found in \cite{falconer1997techniques}. 
In this paper, in addition to $\essudimH$ and $\essldimH$ (see \cite{falconer1997techniques}), we define two new notions of dimension of measures, $\udimH$ and $\ldimH$, as follows: 

\begin{dfn}
	Let $(X, d_X)$ be a metric space and $\mu_X$ be a Borel probability measure on $X$.  
	We define 
	\begin{align*}
		\udimH (\mu_X) & \coloneqq \sup_{x \in X} \liminf_{r \to 0}  \frac{\log \mu_X (B_r (x))}{\log r}, 
		\\
		\ldimH (\mu_X) & \coloneqq \inf_{x \in X} \liminf_{r \to 0} \frac{\log \mu_X (B_r (x))}{\log r},
		\\
		\essudimH (\mu_X) & \coloneqq \esssup \displaylimits_{x \in X} \liminf_{r \to 0}  \frac{\log \mu_X (B_r (x))}{\log r}, 
		\\
		\essldimH (\mu_X) & \coloneqq \essinf \displaylimits_{x \in X} \liminf_{r \to 0}  \frac{\log \mu_X (B_r (x))}{\log r}, 
	\end{align*}
	where $B_r (x) \coloneqq \{ z \in X \mid d_X (z, x) \leq r \}$. 
\end{dfn}
In particular, in \cite{falconer1997techniques},  $\essldimH$ and $\essudimH$ are called the Hausdorff dimension of measures and the upper Hausdorff dimension of measures, respectively.  

	Metric measure geometry originated from the theory of convergence and collapse of Riemannian manifolds.
	In general, the limit space of a sequence of Riemannian manifolds is not necessarily a manifold. 
	It is a metric space or a metric measure space depending on the topology that is considered.  
	Here, a triple $(X, d_X, \mu_X)$ is a \textit{metric measure space} or an \textit{mm-space} for short if $(X, d_X)$ is a complete separable metric space and $\mu_X$ is a Borel probability measure on $X$.  
	We write $X$ for it when no confusion can arise. 
	In this paper, we treat measured Gromov--Hausdorff convergence and box convergence as notions of convergence for Riemannian manifolds or, more generally, for metric measure spaces. 
	Measured Gromov--Hausdorff convergence was introduced in \cite{Fukaya_1987}, and box convergence in \cite{Gromov_greenbook}. 
	We note that both notions of convergence are metrizable. 
	The following proposition holds. 
	
	\begin{prop}[{\cite [Remark 4.34]{Shioya_2016}}] \label{prop: mgh and box}
		If a sequence of compact mm-spaces $\{ X_n \}_{n \in \mathbb{N}}$ measured Gromov--Hausdorff converges to a compact mm-space $X$, then $\{ X_n \}_{n \in \mathbb{N}}$ box converges to $X$. 
	\end{prop}
	
	This implies that all properties that hold for box convergent sequences of compact mm-spaces also hold for measured Gromov--Hausdorff convergent sequences.  
	
	Burago, Gromov, and Perelman proved in \cite{Burago_1992} that the Hausdorff dimension is lower semicontinuous with respect to the Gromov--Hausdorff distance on the set of all compact Alexandrov spaces satisfying some geometric conditions. 
	Also, in \cite{Ma:2019aa}, the semicontinuity of $\essudimH$ and $\essldimH$ was studied under setwise convergence, which is a notion of convergence for probability measures.  
	If we do not impose any conditions on a sequence of mm-spaces, they are not necessarily semicontinuous. 
	Therefore, in this paper, we consider the following two conditions, U($D$) and L($D$). 

\begin{dfn}
	Let $0 < D \leq 1$. 
	We say that an mm-space $X$ satisfies the condition U($D$)  (respectively,  the condition L($D$)) if the equality
	\begin{align*}
		\lim_{r \to 0} \frac{\log \mu_X (B_r (x))}{\log r} &= \sup_{0 < r < D} \frac{\log \mu_X (B_r (x))}{\log r} \\
		& \left( \text{resp.} =  \inf_{0 < r < D} \frac{\log \mu_X (B_r (x))}{\log r}\right)
	\end{align*}
	holds for all $x \in X$.
\end{dfn}

	We write $\mathcal{X}_{\mathrm{U}(D)}$ and $\mathcal{X}_{\mathrm{L}(D)}$ for the sets of all mm-spaces satisfying the conditions U($D$) and L($D$), respectively. 
	We observe that $\mathcal{X}_{\mathrm{U}(D)}$ and $\mathcal{X}_{\mathrm{L}(D)}$ are not necessarily closed sets with respect to the box distance.  
	For further details, see Examples \ref{ex: not closed under UD} and \ref{ex: not closed under LD}. 

	Under either the condition U($D$) or the condition L($D$), we study the semicontinuity of  $\udimH$, $\ldimH$, $\essudimH$, and $\essldimH$ with respect to measured Gromov--Hausdorff convergence and box convergence and obtain the following results. 
	In what follows, fix $0 < D \leq 1$. 
 
 \begin{thm} \label{thm: lowersemi, box}
 	Let $\{X_n \}_{n \in \mathbb{N}}$ be a sequence of mm-spaces satisfying the condition U($D$). 
	If $\{X_n \}_{n \in \mathbb{N}}$ box converges to an mm-space $X$, we have
	\begin{align}
		\label{thm: eq: U, udim, box, lower}
		\udimH (\mu_X) &\leq \liminf_{n \to \infty} \udimH (\mu_{X_n}), \\
		\label{eq: thm: U, ldimH, box, lower}
		\ldimH (\mu_X) & \leq \liminf_{n \to \infty} \ldimH (\mu_{X_n}), 
		\\
		\label{thm: eq: U, essudim, box, lower}
		\essudimH (\mu_X)& \leq \liminf_{n \to \infty} \essudimH (\mu_{X_n}). 
	\end{align}
 \end{thm}
 
Proposition \ref{prop: mgh and box} and Theorem \ref{thm: lowersemi, box} imply the following. 

 \begin{cor}\label{cor: lowersemi, mGH}
 	If a sequence of compact mm-spaces $\{ X_n \}_{n \in \mathbb{N}}$ satisfying the condition U($D$) measured Gromov--Hausdorff converges to a compact mm-space $X$, then 
		\eqref{thm: eq: U, udim, box, lower},
		\eqref{eq: thm: U, ldimH, box, lower}, 
		and \eqref{thm: eq: U, essudim, box, lower}
	hold. 
 \end{cor}

 \begin{thm}\label{thm: uppersemi, box}
 	Let $\{ X_n \}_{n \in \mathbb{N}}$ be a sequence of mm-spaces satisfying the condition L($D$). 
	If $\{ X_n \}_{n \in \mathbb{N}}$ box converges to an mm-space $X$, we have
		\begin{align}
			\label{thm: eq: B, lowerdim, box, upper}
		\limsup_{n \to \infty} \ldimH (\mu_{X_n}) & \leq \ldimH (\mu_X), \\
			\label{thm: eq: B, esslowerdim, box, upper}
		\limsup_{n \to \infty} \essldimH (\mu_{X_n}) & \leq \essldimH (\mu_X). 
	\end{align}
 \end{thm}
 
 Combining Proposition \ref{prop: mgh and box} and Theorem \ref{thm: uppersemi, box}, we obtain the following. 
  
 \begin{cor}\label{cor: uppersemi, mGH}
 	If a sequence of compact mm-spaces $\{ X_n \}_{n \in \mathbb{N}}$ satisfying the condition L($D$) measured Gromov--Hausdorff converges to a compact mm-space $X$, then \eqref{thm: eq: B, lowerdim, box, upper} and \eqref{thm: eq: B, esslowerdim, box, upper} hold.
 \end{cor}

 \begin{thm}\label{thm: uppersemi, mGH}
 	Suppose that a sequence of compact mm-spaces $\{ X_n \}$ satisfying the condition L($D$) measured Gromov--Hausdorff converges to a compact mm-space $X$.  
	Then we have
	\begin{align}
			\label{thm: eq: B, udimH, mGH, upper}
		\limsup_{n \to \infty} \udimH (\mu_{X_n}) & \leq \udimH (\mu_X).
	\end{align}
 \end{thm}

 We summarize the results in the following table. 
 In the table, “No” means that there are counterexamples. 
We give them in Section \ref{sec: example}. 
 \renewcommand{\arraystretch}{1.25}

\vspace{0.1cm}
\noindent
\begin{center}
\begin{tabular}{|c|c|c|c|c|} \hline
 \multicolumn{5}{|c|}{Condition U($D$)} \\ \hline
 & $\square$-lower & $\square$-upper & mGH-lower & mGH-upper \\ \hline
   $\udimH$ &  Yes & No & Yes & No  \\ \hline
   $\ldimH $ & Yes & No	& Yes & No  \\ \hline
   $\essudimH$  &  Yes & No & Yes & No \\ \hline
   $\essldimH$ & No & No & No & No \\ \hline
    \multicolumn{5}{|c|}{Condition L($D$)} \\ \hline
    $\udimH$ & No & No & No & Yes \\ \hline
     $\ldimH $ & No & Yes & No & Yes \\ \hline
   $\essudimH$  & No & No & No & No \\ \hline
   $\essldimH$  & No & Yes & No & Yes \\ \hline
\end{tabular}
\end{center}

\begin{ack}
	The author would like to thank Professor Takashi Shioya for his valuable comments and suggestions. 
	He also thanks Shigeaki Yokota and Toshiaki Miyamoto for many stimulating discussions. 
	He used ChatGPT-6 Astra to prove Lemma \ref{lem: estimate ball} and Examples \ref{ex: counterexample essldimH} and \ref{ex: counterexample, essudimH}. 
\end{ack}

\section{Preliminaries} \label{sec: preli}
	For further details in this section, see \cite{Gromov_greenbook, Shioya_2016}. 
	
\subsection{Box convergence}
	
	\begin{dfn}[mm-isomorphic]
		Let $X$ and $Y$ be mm-spaces. 
		We say that $X$ is {\it mm-isomorphic} to $Y$ if there exists an isometry $f \colon \supp \mu_X \to \supp \mu_Y$ such that $f_* \mu_X = \mu_Y$, where $f_* \mu_X$ is the push-forward measure of $\mu_X$ by $f$. 
	\end{dfn}
	
	The mm-isomorphism relation is an equivalence relation. 
	Let $\mathcal{X}$ denote the set of all mm-isomorphism classes of mm-spaces. 
	
\begin{rmk}
		$\mathcal{X}$ is expressed as a set. 
\end{rmk}

In this paper, we assume that $X = \supp \mu_X$ for an mm-space $X$. 

\begin{dfn}[Distortion]
	Let $(X, d_X)$ and $(Y, d_Y)$ be metric spaces. 
	The {\it distortion}  of a subset $S \subset X \times Y$ is defined by 
	\[
		 \dis (S) \coloneqq 
		\begin{cases*}
			\sup \{ |d_X (x, x^{\prime}) - d_Y (y, y^{\prime})| \mid (x, y), (x^{\prime}, y^{\prime}) \in S \} & if $S \neq \emptyset$,  \\
			0 & if $S = \emptyset$.
		\end{cases*}
	\]
\end{dfn}

\begin{dfn}[Transport plan]
	Suppose that $X$ and $Y$ are topological spaces. 
		Let $\mu_X$ and $\mu_Y$ be finite Borel measures on $X$ and $Y$, respectively. 
		Let $p_X \colon X \times Y \to X$ and $p_Y \colon X \times Y \to Y$ be the projection maps. 
		We call a Borel measure $\pi$ on $X \times Y$ a {\it transport plan} between $\mu_X$ and $\mu_Y$ if $\pi$ satisfies $(p_X)_* \pi = \mu_X$ and $(p_Y)_* \pi = \mu_Y$. 
		We write $\Pi (\mu_X, \mu_Y)$ for the set of all transport plans between $\mu_X$ and $\mu_Y$. 
 	\end{dfn}

\begin{dfn}[Box distance, \cite{Nakajima:2022aa}]
	For $X, Y \in \mathcal{X}$, we define
	\[
		\square (X, Y) \coloneqq \inf_{\pi \in \Pi (\mu_X, \mu_Y)} \inf_{S \subset X \times Y \text{: Borel}} \max \{ \dis (S), 1 - \pi (S) \}. 
	\]
	We call it the {\it box distance} between $X$ and $Y$. 
\end{dfn}

\begin{thm} \label{thm: box distance complete separable}
	$(\mathcal{X}, \square)$ is a complete separable metric space. 
\end{thm}

The topology on $\mathcal{X}$ induced by the box distance is called the \textit{box topology}. 

\begin{dfn}[Prohorov distance]
	Let $(X, d_X)$ be a metric space and $\mathcal{M} (X)$ be the set of all Borel probability measures on $X$. 
	For $\mu, \nu \in \mathcal{M} (X)$, we define
		\[
			d_{\mathrm{P}} (\mu, \nu) \coloneqq \inf \{ \varepsilon > 0 \mid \text{for any Borel set $A$ of $X$,  $\mu (U_{\varepsilon} (A)) \geq \nu (A) - \varepsilon$} \}, 
		\]
		where we set
		\[
			U_{\varepsilon} (A) \coloneqq 
			\begin{cases*}
				\{ x \in X \mid d_X (x, A) < \varepsilon \} & if $A \neq \emptyset$,   \\
				\emptyset & if $A = \emptyset$, 
			\end{cases*}
		\]
		and $d_X (x, A) \coloneqq \inf_{y \in A} d_X (x, y)$. 
		We refer to $d_{\mathrm{P}}$ as the {\it Prohorov distance}. 
	\end{dfn}
	
	\begin{thm}[{\cite[\S 6]{Billingsley_1999}}]
		Let $X$ be a separable metric space. 
		Then $d_{\mathrm{P}}$ is a metric on $\mathcal{M} (X)$ and a sequence in $\mathcal{M} (X)$ converges in the sense of the Prohorov distance if and only if it converges weakly.  
	\end{thm}

\begin{prop} \label{prop: box and prokhorov}
	Let $(X, d_X)$ be a complete separable metric space. For $\mu, \nu \in \mathcal{M} (X)$, we have
	\[
		\square ( (X, \mu), (X, \nu)) \leq 2 d_{\mathrm{P}} (\mu, \nu). 
	\]
\end{prop}

Proposition \ref{prop: box and prokhorov} shows that we can estimate the box distance by the Prohorov distance when the two measures are defined on the same metric space.
In general, the following statement is known. 

\begin{lem} [{\cite[Lemma 4.2]{kazukawadoctor}}]\label{prop: box converge and embedding}
	Let $X$ be an mm-space and $\{ X_n\}_{n \in \mathbb{N}}$ be a sequence of mm-spaces.  
	If $\{ X_n \}_{n \in \mathbb{N}}$ box converges to $X$ as $n \to \infty$, 
	then there exist a complete separable metric space $Y$ and isometric embeddings $\iota \colon X \to Y$, $\iota_n \colon X _n \to Y$ ($n \in \mathbb{N}$) such that $(\iota_n)_* \mu_{X_n}$ converges weakly to $\iota_* \mu_X$ and for $x \in X$, there exist $x_n \in X_n$ ($n \in \mathbb{N}$) such that $\{ \iota_n (x_n) \}_{n \in \mathbb{N}}$ converges to $\iota (x)$. 
\end{lem}

Next, we explain the relation between the box distance and an $\varepsilon$-mm-isomorphism. 

\begin{dfn}[$\varepsilon$-mm-isomorphism]
	Let $X$, $Y$ be mm-spaces and $\varepsilon \geq 0$. 
	We call a Borel measurable map $f \colon X \to Y$ an \textit{$\varepsilon$-mm-isomorphism} if there exists a Borel set $\tilde{X} \subset X$ such that
	\begin{itemize}
		\item $\mu_X (\tilde{X}) \geq 1 - \varepsilon$, 
		\item $|d_X (x_1, x_2) - d_Y (f (x_1), f (x_2))| \leq \varepsilon$ for $x_1, x_2 \in \tilde{X}$, 
		\item $d_{\mathrm{P}} (f_* \mu_X, \mu_Y) \leq \varepsilon$. 
	\end{itemize}
	We say that $\tilde{X}$ is a \textit{non-exceptional domain} of $f$. 
\end{dfn}

\begin{prop} \label{prop: box distance and epsilon-mm-ismom}
	Let $X$, $Y$ be mm-spaces and $\varepsilon \geq 0$.
	The following statements hold.
	\begin{enumerate}
		\item If $\square (X, Y) < \varepsilon$, then there exists a $3 \varepsilon$-mm-isomorphism $f \colon X \to Y$. 
		\item If there exists an $\varepsilon$-mm-isomorphism $f \colon X \to Y$, then $\square (X, Y) \leq 3 \varepsilon$. 
	\end{enumerate}

\end{prop}

\subsection{Measured Gromov--Hausdorff convergence}	

	In this subsection, we discuss measured Gromov--Hausdorff convergence, which is a notion of convergence for compact mm-spaces. 
	
	\begin{dfn}[$\varepsilon$-isometry]
		Let $\varepsilon \geq 0$ and $(X, d_X)$, $(Y, d_Y)$ be metric spaces. 
		We say that a map $f \colon X \to Y$ is an \textit{$\varepsilon$-isometry} if it satisfies  
		\[
			\sup_{x, x^\prime \in X} |d_X (x, x^\prime) - d_Y (f (x), f (x^\prime))| \leq \varepsilon
		\]
		and $Y = B_{\varepsilon} (f (X)) \coloneqq \{ y \in Y \mid d_Y (y, f(X)) \leq \varepsilon \}$.
	\end{dfn}

	\begin{dfn}[Measured Gromov--Hausdorff convergence] \label{dfn: mgh convergence}
		We say that a sequence of compact mm-spaces $\{ X_n \}_{n \in \mathbb{N}}$ \textit{measured Gromov--Hausdorff converges} to a compact mm-space $X$  if there exist a sequence of positive real numbers $\{ \varepsilon_n \}_{n \in \mathbb{N}}$ with $\varepsilon_n \to 0$ as $n \to \infty$ and Borel measurable $\varepsilon_n$-isometries $f_n \colon X_n \to X$ such that $\{ (f_n)_* \mu_{X_n} \}_{n \in \mathbb{N}}$ converges weakly to $\mu_X$. 
	\end{dfn}
	
	Proposition \ref{prop: mgh and box} follows immediately from Proposition \ref{prop: box distance and epsilon-mm-ismom} and Definition \ref{dfn: mgh convergence}. 
	
\section{Proof of the main theorems}
\subsection{Condition U($D$)}

\begin{lem} \label{lem: esssup and sup}
	Let $X$ be an mm-space and $f \colon X \to [0, + \infty]$ be a Borel measurable map. 
	Then there exist Borel sets $N_1, N_2 \subset X$ such that $\mu_X (N_1) = \mu_X (N_2)= 0$ and 
	\begin{align}
		\label{eq: lem: esssup and sup}
		\esssup f &= \sup_{x \in X \setminus N_1} f (x), 
		\\
		\label{eq: lem: essinf and inf}
		\essinf f & = \inf_{x \in X \setminus N_2} f (x). 
	\end{align}
\end{lem}

\begin{proof}
	First, we show \eqref{eq: lem: esssup and sup}. 
	\eqref{eq: lem: essinf and inf} may be proved in much the same way. 
	Let $\alpha \coloneqq \esssup f = \inf \{ a \in \mathbb{R} \mid \mu_X (f > a) = 0 \}$. 
	If $\alpha = + \infty$, it is clear that $\sup_{x \in X \setminus N_1} f (x) = + \infty$ by setting $N_1 = \emptyset$. 	We consider the case $\alpha < + \infty$. 
	Let $N_1 \coloneqq \{ f > \alpha \}$. 
	Since $\{ f > \alpha \} = \bigcup_{n \in \mathbb{N}} \{ f > \alpha + n^{-1} \}$ and $\mu_X (f > \alpha + n^{-1}) = 0$, we see $\mu_X (N_1) = 0$. 
	Therefore we obtain $\sup_{x \in X \setminus N_1} f (x) \leq \alpha$ because $f (x) \leq \alpha$ whenever $x \in X \setminus N_1$. 
	
	Next, we prove the other inequality. 
	Fix $\varepsilon > 0$. 
	We note that $\mu_X ( \{ f > \alpha - \varepsilon \}) > 0$. 
	$ \{ f > \alpha - \varepsilon \} \setminus N_1$ is nonempty because $\mu_X ( \{ f > \alpha - \varepsilon \} \setminus N_1) = \mu_X ( \{ f > \alpha - \varepsilon \}) - \mu_X (N_1) > 0$. 
	Taking $x \in  \{ f > \alpha - \varepsilon \} \setminus N_1 \subset X \setminus N_1$, we have $f (x) > \alpha - \varepsilon$. 
	This yields $\sup_{x \in X \setminus N_1} f (x) > \alpha - \varepsilon$. 
	Since $\varepsilon > 0$ is arbitrary, we find $\sup_{x \in X \setminus N_1} f (x) \geq \alpha$. 
	This completes the proof.
\end{proof}

For an mm-space $X$, we note that the map
\[
	X \ni x \mapsto \liminf_{r \to 0} \frac{\log \mu_X (B_r (x))}{\log r} \in [0, + \infty]
\]
is a Borel measurable map. 

\begin{lem}
	Let $X$ be an mm-space and $x \in X$. 
	Then we have
	\[
		\liminf_{r \to 0} \frac{\log \mu_X (B_r (x))}{\log r} = \liminf_{r \to 0} \frac{\log \mu_X (U_r (x))}{\log r}. 
	\]
\end{lem}

\begin{proof}
	It is clear that the inequality $\leq$ holds. 
	We show the other inequality. 
	Fix $0 < r < 1/2$. 
	Let $r^\prime = 2r$. 
	Using $\log \mu_X (B_r (x)) / \log r \geq \log \mu_X (U_{2r} (x)) / \log r$, we have
	\[
		\frac{\log \mu_X (B_r (x))}{\log r} \geq \frac{\log \mu_X (U_{r^\prime} (x))}{\log r^{\prime}} \cdot \frac{\log r^\prime}{\log r^{\prime} - \log 2}. 
	\]
	Letting $r \to 0$,  we obtain $r^\prime \to 0$ and $\liminf_{r \to 0} \log \mu_X (B_r (x)) / \log r \geq \liminf_{r \to 0} \log \mu_X (U_r (x)) / \log r $. 
	This completes the proof. 
\end{proof}

This immediately implies the following.

\begin{cor} \label{cor: another definition of essldimH }
	Let $X$ be an mm-space. Then we have
	\[
		\essldimH (\mu_X) = \essinf \displaylimits_{x \in X} \liminf_{r \to 0}  \frac{\log \mu_X (U_r (x))}{\log r}. 
	\]
\end{cor}

\begin{proof}[Proof of Theorem \ref{thm: lowersemi, box}]
	First, we show \eqref{thm: eq: U, essudim, box, lower}. 	
	Since \eqref{thm: eq: U, udim, box, lower} can be proved in the same way, we omit its proof.	
	By Lemma \ref{lem: esssup and sup} and the condition U($D$), we can choose Borel sets $N_n \subset X_n$ and $N \subset X$ satisfying 
	\begin{align*}
		\essudimH (\mu_{X_n}) &
		= \sup_{y_n \in X_n \setminus N_n} \sup_{0 < r < D} \frac{\log \mu_{X_n} (B_r (y_n))}{\log r}, 
		\
		\\
		\essudimH (\mu_X) & =
		\sup_{y \in X \setminus N} \liminf_{r \to 0} \frac{\log \mu_X (B_r (y))}{\log r}.
	\end{align*}
	
	Fix $x \in X \setminus N$, $\varepsilon > 0$, and $0 < r < D$. 
	Since $\{ X_n \}_{n \in \mathbb{N}}$ box converges to $X$ and Lemma \ref{prop: box converge and embedding} holds, 
	$X$ and $X_n$, $n \in \mathbb{N}$ are isometrically embedded into some complete separable metric space $Z$ 
	and
	there exist $\bar{x}_n \in X_n$ and $N \in \mathbb{N}$ such that $d_{\mathrm{P}} (\mu_{X_n}, \mu_X)< \varepsilon$ and $d_Z (\bar{x}_n, x) < \varepsilon /2$ for each $n \geq N$.  
	As $X_n \setminus N_n$ is dense in $X_n$, we can find $x_n \in X_n \setminus N_n$ with $d_{X_n} (x_n, \bar{x}_n) < \varepsilon / 2$. 
	Hence we get $d_Z (x_n, x) < \varepsilon$. 
	For $n \geq N$, we have
	\[
	 	\mu_X (B_{r + 2 \varepsilon} (x)) \geq \mu_{X_n} (B_r (x_n)) - \varepsilon. 
	\]
	Therefore we obtain
	 \begin{align*}
		\frac{\log (\mu_X (B_{r + 2 \varepsilon} (x)) + \varepsilon)}{\log r} & \leq \frac{\log \mu_{X_n} (B_r (x_n))}{\log r}
		\leq
		\sup_{0 < r < D}
		\frac{\log \mu_{X_n} (B_r (x_n))}{\log r} \\
		&\leq
		\sup_{x_n \in X_n \setminus N_n} \sup_{0 < r < D}
		\frac{\log \mu_{X_n} (B_r (x_n))}{\log r}
		=
		\essudimH (\mu_{X_n}). 
	\end{align*}
	Letting $n \to \infty$ and then $\varepsilon \to 0$, we see
	\[
		\frac{\log \mu_X (B_r (x))}{\log r}
		\leq
		\liminf_{n \to \infty} \essudimH (\mu_{X_n}). 
	\]
	Since $x \in X \setminus N$ is arbitrary, we have \eqref{thm: eq: U, essudim, box, lower}. 
	
	Next, we prove \eqref{eq: thm: U, ldimH, box, lower}.  
	Without loss of generality, we may assume that 
	\[
		a \coloneqq \liminf_{n \to \infty} \ldimH (\mu_{X_n}) = \lim_{n \to \infty} \ldimH (\mu_{X_n}) < + \infty
	\]
	 and $\ldimH (\mu_{X_n})$ is finite. 
	For simplicity, we set $\alpha_n \coloneqq \ldimH (\mu_{X_n})$. 
	It suffices to show that there exists $x \in X$ such that
	\begin{align} \label{eq: thm: U, ldimH, box, lower-4}
		\frac{\log \mu_X (U_r (x))}{\log r} \leq \liminf_{n \to \infty} \ldimH (\mu_{X_n})
	\end{align}
	for $0 < r < D$. 
	As $\{ X_n \}_{n \in \mathbb{N}}$ box converges to $X$, we choose a sequence of positive numbers $\{ \varepsilon_n \}_{n \in \mathbb{N}}$ with $\varepsilon_n \to 0$ and an $\varepsilon_n$-mm isomorphism $f_n \colon X_n \to X$. 
	We write $\tilde{X}_n \subset X_n$ for a non-exceptional domain of $f_n$. 
	By taking $n \in \mathbb{N}$ large enough, we may assume that $\beta_n \coloneqq  \varepsilon_n^{\frac{1}{\alpha_n + \varepsilon_n + 1}} < D$. 
	
	\begin{claim} \label{claim: U, essldimH, lower-1}
		For $x_n \in X_n \setminus B_{\beta_n} (\tilde{X}_n)$, we have
		\[
			\alpha_n + \varepsilon_n \leq \sup_{0 < r < D} \frac{\log \mu_{X_n} (B_r (x_n))}{\log r}. 
		\]
	\end{claim}
	
	\begin{proof}[Proof of Claim \ref{claim: U, essldimH, lower-1}]
		Fix $x_n \in X_n \setminus B_{\beta_n} (\tilde{X}_n)$. 
		Since $B_{\beta_n} (x_n) \subset X_n \setminus \tilde{X}_n$ and $\mu_{X_n} (X_n \setminus \tilde{X}_n) \leq \varepsilon_n$, we obtain
		\begin{align*}
			\alpha_n + \varepsilon_n \leq \frac{\log \mu_{X_n} (B_{\beta_n} (x_n))}{\log \beta_n}
			\leq
			\sup_{0 < r < D} \frac{\log \mu_{X_n} (B_r (x_n))}{\log r}. 
		\end{align*}
	\end{proof}
	
	Claim \ref{claim: U, essldimH, lower-1} and Lemma \ref{lem: esssup and sup} imply that there exists $x_n \in B_{\beta_n} (\tilde{X}_n)$ satisfying
	\begin{align} \label{eq: thm: U, essdimH, lower, box-1}
		\alpha_n \leq \sup_{0 < r < D} \frac{\log \mu_{X_n} (B_r (x_n))}{\log r} < \alpha_n + \varepsilon_n . 
	\end{align}
	We define $z_n \in X_n$ as follows: 
	If $x_n \in \tilde{X}_n$, let $z_n \coloneqq x_n$. 
	Otherwise, there exists $x^\prime_n \in \tilde{X}_n$ with $d_{X_n} (x^\prime_n, x_n) < 2 \beta_n$.
	We set $z_n \coloneqq x^\prime_n$. 

	Take $0 < r < D$ and $0 < \varepsilon < r$ arbitrarily. 
	There exists $N (r, \varepsilon) \in \mathbb{N}$ such that $|\alpha_n - a| < \varepsilon /2$ and $2 (\beta_n +  \varepsilon_n) < \min \{\varepsilon, (r- \varepsilon)^{a + \varepsilon} \}$ for every $n \geq N (r, \varepsilon)$. 
	For $n \geq N(r, \varepsilon)$,  we obtain
	\begin{align}
		\mu_{X_n} (B_{r-2 \beta_n - 2 \varepsilon_n} (x_n))
		&
		\leq \mu_{X_n} (B_{r- 2 \varepsilon_n} (z_n)) \nonumber \\
		&\leq \mu_{X_n} (B_{r - 2 \varepsilon_n } (z_n) \cap \tilde{X}_n) + \mu_{X_n} (X_n \setminus \tilde{X}_n) \nonumber
		\\
		&
		\leq \mu_{X_n} (B_{r- 2 \varepsilon_n} (z_n) \cap \tilde{X}_n) +  \varepsilon_n. \nonumber
		\intertext{Since $B_{r- 2 \varepsilon_n} (z_n) \cap \tilde{X}_n \subset f_n^{-1} (B_{r - \varepsilon_n} (f_n (z_n)))$ and $d_{\mathrm{P}} ((f_n)_*\mu_{X_n}, \mu_X) \leq \varepsilon_n$, we have}
		&
		\leq (f_n)_* \mu_{X_n} (B_{r - \varepsilon_n} (f_n (z_n))) +  \varepsilon_n \nonumber
		\\
		&
		\label{eq: thm: U, essdimH, lower, box-2}
		\leq \mu_X (B_r (f_n (z_n))) + 2 \varepsilon_n. 
	\end{align}
	Combining \eqref{eq: thm: U, essdimH, lower, box-1} and \eqref{eq: thm: U, essdimH, lower, box-2}, we see 
	\begin{align} \label{eq: thm: U, essldim, box, lower-3}
		0 < (r - \varepsilon)^{a + \varepsilon} - 2 \varepsilon_n \leq \mu_X (B_r (f_n (z_n)))
	\end{align}
	for $n \geq N(r, \varepsilon)$. 
	
	\begin{claim} \label{claim: U, essldimH, lower-2}
		$\{ f_n (z_n) \}_{n \in \mathbb{N}}$ has a convergent subsequence. 
	\end{claim}
	
	\begin{proof}[Proof of Claim \ref{claim: U, essldimH, lower-2}]
		Suppose that $\{ f_n (z_n)\}_{n \in \mathbb{N}}$ has no convergent subsequence. 
		This means that $\{ f_n (z_n) \}_{n \in \mathbb{N}}$ is not a totally bounded set. 
		Therefore there exist $\delta > 0$ and a subsequence $\{ f_{n (k)} (z_{n (k)}) \}_{k \in \mathbb{N}}$ of $\{ f_n (z_n) \}_{n \in \mathbb{N}}$ with $d_X (f_{n(k)} (z_{n(k)}), f_{n (l)} (z_{n (l)})) \geq \delta$ for all distinct numbers $k, l \in \mathbb{N}$. 
		For $0 < r < \min \{ D, \delta /3\}$ and $0 < \varepsilon < r$, we see that \eqref{eq: thm: U, essldim, box, lower-3} holds whenever $n \geq N(r, \varepsilon)$. 
		Since $\varepsilon_n \to 0$, the inequality
		\[
			2 \varepsilon_n < \frac{1}{2} (r - \varepsilon)^{a + \varepsilon}
		\]
		holds for all sufficiently large $n \in \mathbb{N}$. 
		Therefore we obtain
		\[
			\mu_X (B_r (f_n(k)(z_{n(k)}))) \geq \frac{1}{2} (r - \varepsilon)^{a + \varepsilon} > 0
		\]
		for all sufficiently large $k \in \mathbb{N}$. 
		For all sufficiently large $k, l \in \mathbb{N}$ with $k \neq l$, we note that $B_r (f_{n(k)} (z_{n (k)}))$ and  $B_r (f_{n(l)} (z_{n (l)}))$ are disjoint. 
		This contradicts the fact that $\mu_X$ is a Borel probability measure. 
	\end{proof}

	There is no loss of generality in assuming $\{f_n (z_n) \}_{n \in \mathbb{N}}$ converges to some $x \in X$. 
	Take $n \in \mathbb{N}$ large enough to satisfy $d_X (f_n (z_n), x) < \varepsilon/ 2$. 
	Replacing $r - 2 \beta_n- 2 \varepsilon_n$ by $r -2 \beta_n - 2 \varepsilon_n - \varepsilon$ in \eqref{eq: thm: U, essdimH, lower, box-2}, we have 
	\[
		\mu_{X_n} (B_{r - 2 \beta_n - 2 \varepsilon_n - \varepsilon} (x_n)) \leq \mu_X (U_r (x)) + 2 \varepsilon_n. 
	\]
	Therefore, we see that
	\begin{align*}
		\frac{\log (\mu_X (U_r  (x)) + 2 \varepsilon_n )}{\log (r - 2 \beta_n - 2 \varepsilon_n - \varepsilon)} 
		& \leq
		\frac{\log \mu_{X_n} (B_{r - 2\beta_n - 2 \varepsilon_n - \varepsilon} (x_n))}{\log (r - 2 \beta_n - 2 \varepsilon_n - \varepsilon)} \\
		& \leq
		\sup_{0 < r < D} \frac{\log \mu_{X_n} (B_r (x_n))}{\log r} \leq \alpha_n + \varepsilon_n. 
	\end{align*}
	Letting $n \to \infty$ and then $\varepsilon \to 0$ yields \eqref{eq: thm: U, ldimH, box, lower-4}. 
	This completes the proof. 
\end{proof}

\subsection{Condition L($D$)}

In this subsection, we prove Theorems \ref{thm: uppersemi, box} and \ref{thm: uppersemi, mGH}. 
The main idea of these proofs is the same as that of Theorem \ref{thm: lowersemi, box}. 

\begin{proof}[Proof of Theorem \ref{thm: uppersemi, box}]
	Here, we show \eqref{thm: eq: B, esslowerdim, box, upper}. 
	\eqref{thm: eq: B, lowerdim, box, upper} may be proved in much the same way as \eqref{thm: eq: B, esslowerdim, box, upper}. 
	
	By Lemma \ref{lem: esssup and sup} and the condition L($D$), we find Borel sets  $N_n \subset X_n$, $N \subset X$ with $\mu_{X_n} (N_n) = \mu_X (N) = 0$ and
	\begin{align*}
		\essldimH (\mu_{X_n}) &= \inf_{y_n \in X_n \setminus N_n} \inf_{0 < r < D} \frac{\log \mu_{X_n} (B_r (y_n))}{\log r}, 
		\\
		\essldimH (\mu_X) & =
		\inf_{y \in X \setminus N} \liminf_{r \to 0} \frac{\log \mu_X (B_r (y))}{\log r}. 
	\end{align*}
	
	Fix $x \in X \setminus N$, $0 < r < D$, and $0 < \varepsilon < \min \{ \mu_X (B_r (x)), (D - r)/2 \}$. 
	Since $\{ X_n \}_{n \in \mathbb{N}}$ box converges to $X$, Lemma \ref{prop: box converge and embedding} implies that  
	$X$ and $X_n$, $n \in \mathbb{N}$ are isometrically embedded into some complete separable metric space $Z$
	and there exist $\bar{x}_n \in X_n$ and $N \in \mathbb{N}$ such that  $d_Z (\bar{x}_n, x) < \varepsilon /2$, $d_{\mathrm{P}} (\mu_{X_n}, \mu_X)< \varepsilon$ for $ n \geq N$. 
	As $X_n \setminus N_n$ is dense in  $X_n$, we find $x_n \in X_n \setminus N_n$ with $d_{X_n} (x_n, \bar{x}_n) < \varepsilon / 2$. 
	Thus, we have $d_Z (x_n, x) < \varepsilon$. 
	These inequalities show that 
	\[
		\mu_{X_n} (B_{r + 2 \varepsilon} (x_n)) \geq \mu_X (B_r (x)) - \varepsilon > 0
	\]
	for $n \geq N$. 
	Using the condition L($D$), we obtain 
	\begin{align*}
		\frac{\log \mu_X (B_r (x))}{\log (r + 2 \varepsilon)} 
		&\geq
		 \frac{\log (\mu_{X_n} (B_{r + 2 \varepsilon} (x_n))+ \varepsilon)}{\log (r + 2 \varepsilon)}
		\\
		&\geq
		\frac{\log \mu_{X_n} (B_{r + 2 \varepsilon} (x_n))}{\log (r + 2 \varepsilon)} 
		+
		\frac{\log \left (1 + \frac{\varepsilon}{\mu_{X_n} (B_{r + 2 \varepsilon} (x_n))} \right)}{\log (r + 2 \varepsilon)}
		\\
		& \geq
		\essldimH (\mu_{X_n}) + \frac{\log \left (1 + \frac{\varepsilon}{\mu_X (B_r (x)) - \varepsilon} \right)}{\log (r + 2 \varepsilon)}. 
	\end{align*}
	Letting $n \to \infty$ and then $\varepsilon \to 0$ gives 
	\[
		\frac{\log \mu_X (B_r (x))}{\log r} \geq \limsup_{n \to \infty} \essldimH (\mu_{X_n}). 
	\]
	This completes the proof by taking $r \to 0$. 
\end{proof}

\begin{proof}[Proof of Theorem \ref{thm: uppersemi, mGH}]
	First, we show \eqref{thm: eq: B, udimH, mGH, upper}. 
	By the assumption, we choose a sequence of positive real numbers $\{ \varepsilon_n \}_{n \in \mathbb{N}}$ with $\varepsilon_n \to 0$ and Borel measurable $\varepsilon_n$-isometries $f_n \colon X_n \to X$ such that $(f_n)_* \mu_{X_n}$ converges weakly to $\mu_X$. 
	If $\limsup_{n \to \infty} \udimH (\mu_{X_n}) < + \infty$, we may assume without loss of generality that $\limsup_{n \to \infty} \udimH (\mu_{X_n}) = \lim_{n \to \infty} \udimH (\mu_{X_n})$ and $\udimH (\mu_{X_n}) < + \infty$ for each $n \in \mathbb{N}$.
	For $n \in \mathbb{N}$, there exists $x_n \in X_n$ with
	\[
		- \varepsilon_n + \udimH (\mu_{X_n}) \leq \inf_{0 < r < D} \frac{\log \mu_{X_n} (B_r (x_n))}{\log r}. 
	\]
	Since $X$ is compact, $\{ f_n (x_n) \}_{n \in \mathbb{N}}$ has a convergent subsequence. 
	Passing to a subsequence, we may assume that $\{ f_n (x_n) \}_{n \in \mathbb{N}}$ converges to some $x \in X$. 
	
	Fix $0 < r < D$ and $0 < \varepsilon < \min \{ \mu_{X} (B_r (x)) , (D-r)/ 2 \}$. 
	There exists $N \in \mathbb{N}$ such that $\varepsilon_n + 2\varepsilon + r < D$, $d_X (x, f_n (x_n)) < \varepsilon$, and $d_{\mathrm{P}} ((f_n)_* \mu_{X_n}, \mu_X) < \varepsilon$ whenever $n \geq N$. 
	These imply that 
	\begin{align*}
		(f_n)_* \mu_{X_n} (B_{r + 2 \varepsilon} (f_n (x_n)))
		&\geq
		\mu_X (B_{r + \varepsilon} (f_n (x_n))) - \varepsilon \\
		&\geq
		\mu_X (B_r (x)) - \varepsilon
	\end{align*}
	for $n \geq N$. 
	Since $f_n^{-1} (B_{r + 2 \varepsilon} (f_n (x_n))) \subset B_{r + 2\varepsilon + \varepsilon_n} (x_n)$, we see
	\[
		\mu_{X_n} (B_{r + 2 \varepsilon + \varepsilon_n} (x_n)) 
		\geq
		\mu_X (B_r (x)) - \varepsilon > 0. 
	\]
	We obtain 
	\begin{align*}
		&\frac{\log \mu_{X} (B_r (x))}{\log (r + 2 \varepsilon + \varepsilon_n)}
		\geq
		\frac{\log ( \mu_{X_n} (B_{r + 2 \varepsilon + \varepsilon_n} (x_n)) + \varepsilon )}{\log (r + 2 \varepsilon + \varepsilon_n)}
		\\
		&= \frac{\log \mu_{X_n} (B_{r + 2 \varepsilon + \varepsilon_n} (x_n))}{\log (r + 2 \varepsilon + \varepsilon_n)}
		+ 
		\frac{\log \left( 1 + \frac{\varepsilon}{\mu_{X_n}(B_{r + 2 \varepsilon + \varepsilon_n} (x_n))}\right)}{\log (r + 2 \varepsilon + \varepsilon_n)}
		\\
		&
		\geq
		\inf_{0 < r < D} \frac{\log \mu_{X_n} (B_r (x_n))}{\log r} +\frac{\log \left( 1 + \frac{\varepsilon}{\mu_X(B_r (x)) - \varepsilon}\right)}{\log (r + 2 \varepsilon + \varepsilon_n)} \\
		&
		\geq
		- \varepsilon_n + \udimH  (\mu_{X_n})
		+
		\frac{\log \left( 1 + \frac{\varepsilon}{\mu_X(B_r (x)) - \varepsilon}\right)}{\log (r + 2 \varepsilon + \varepsilon_n)}. 
	\end{align*}
	Letting $n \to \infty$ and then $\varepsilon \to 0$, we see
	\[
		\limsup_{n \to \infty} \udimH (\mu_{X_n}) \leq \frac{\log \mu_X (B_r (x))}{\log r}. 
	\]
	Letting $r \to 0$ yields \eqref{thm: eq: B, udimH, mGH, upper}. 
		
	Next, we show \eqref{thm: eq: B, udimH, mGH, upper} if $\limsup_{n \to \infty} \udimH (\mu_{X_n}) = + \infty$. 
	Fix $R > 0$. 
	Passing to a subsequence, we assume that for every $n \in \mathbb{N}$, there exists $x_n \in X_n$ such that
	\[
		\liminf_{r \to 0} \frac{\log \mu_{X_n} (B_r (x_n))}{\log r} \geq R. 
	\]	
	By a similar argument, we obtain $R \leq \udimH (\mu_X)$. 
	Since $R > 0$ is arbitrary, we have $\udimH (\mu_X) = + \infty$. 
	This completes the proof. 
\end{proof}

\section{Examples} \label{sec: example}
\subsection{Condition U($D$)}

The following example provides a sequence of mm-spaces satisfying the condition U($D$). 
\begin{ex}
	Let $F_1 \colon [0, 1] \to [0, 1], x \mapsto x/3$ and $F_2 \colon [0, 1] \to [0, 1], x \mapsto x/3 + 2/3$. 
	For $n \in \mathbb{N}$, we set $I_n \coloneqq \{ (i_1, \ldots, i_n) \mid i_j = 1, 2, j = 1, \ldots, n \}$ and
	\[
		X_n \coloneqq \bigcup_{(i_1, \ldots, i_n) \in I_n} (F_{i_1} \circ \cdots \circ F_{i_n}) ([0, 1]). 
	\]
	We define a Borel probability measure $\mu_{X_n}$ on $X_n$ by 
	\[
		\mu_{X_n}  \coloneqq \left(\frac{3}{2}\right)^n \sum_{(i_1, \ldots, i_n) \in I_n} \mathcal{L}^1|_{(F_{i_1} \circ \cdots \circ F_{i_n}) ([0, 1])}, 
	\]
	where $ \mathcal{L}^1|_{(F_{i_1} \circ \cdots \circ F_{i_n}) ([0, 1])}$ denotes the restriction of the one-dimensional Lebesgue measure to $(F_{i_1} \circ \cdots \circ F_{i_n})([0, 1])$. 
	Then $(X_n, \mu_{X_n})$ is an mm-space. 
	By {\cite[Theorem 2.8]{falconer1997techniques}},  $\{ \mu_{X_n}\}_{n \in \mathbb{N}}$ converges weakly to some Borel probability measure $\mu_C$ on the Cantor set $C$. 
	It follows from {\cite[Theorem 2.6]{falconer1997techniques}} that $ \{ X_n \}_{n \in \mathbb{N}}$ Gromov--Hausdorff converges to the Cantor set. 
	Therefore, we see that $\{ (X_n, \mu_{X_n}) \}_{n \in \mathbb{N}}$ measured Gromov--Hausdorff converges to $(C, \mu_C)$. 
	
	We show that $\{ (X_n, \mu_{X_n}) \}_{n \in \mathbb{N}, n \geq 10}$ satisfies the condition U($2/3^7$). 
	Fix $n \in \mathbb{N}$ with $n \geq 10$ and $x_n \in X_n$. 
	First, we prove 
	\[
		\lim_{r \to 0} \frac{\log \mu_{X_n} (B_r (x_n))}{ \log r} = 1 , \hspace{0.5cm} \sup_{0 < r < 2/3^7} \frac{\log \mu_{X_n} (B_r (x_n)) } { \log r} = 1.
	\] 
	There exists $(i_1, \ldots, i_n) \in I_n$ with $x_n \in (F_{i_1} \circ \cdots \circ F_{i_n}) ([0, 1])$. 
	We note that $(F_{i_1} \circ \cdots \circ F_{i_n}) ([0, 1])$ is a bounded closed interval. 
	Let $a$ and $b$ be the left endpoint and the right endpoint of $(F_{i_1} \circ \cdots \circ F_{i_n}) ([0, 1])$, respectively. 
	We set $r^\prime \coloneqq \min \{ b-x_n, x_n-a \}$. 
	If $r^{\prime} = 0$ and $0 < r < 3^{-n}$, then we have $\mu_{X_n} (B_r (x_n)) = 3^nr / 2^n$. 
	This implies that
	\[
		\frac{\log \mu_{X_n} (B_r (x_n))}{\log r} = 1 + \frac{\log \frac{3^n}{2^n}}{\log r}. 
	\]
	If $r^\prime > 0$ and $0 < r < \min \{2/ 3^7, r^\prime \}$, then we obtain $\mu_{X_n} (B_r (x_n)) = 3^n r/ 2^{n-1}$. 
	This yields
	\[
		\frac{\log \mu_{X_n} (B_r (x_n))}{\log r} = 1 + \frac{\log \frac{3^n}{2^{n-1}}}{\log r}. 
	\]
	Therefore, it suffices to show that $\log \mu_{X_n} (B_r (x_n)) / \log r \leq 1$ for $0 < r < 2/3^7$. 
	
	\begin{itemize}
		\item If $0 < r \leq 1/ 3^n$, we have $\mu_{X_n} (B_r (x_n)) \geq 3^n r/ 2^n$. 
		Thus we see
		\[
			\frac{\log \mu_{X_n} (B_r (x_n))}{\log r} \leq 1 + \frac{\log \frac{3^n}{2^n}}{\log r} < 1. 
		\]
		\item If $1/ 3^n \leq r < 2/3^{n-1}$, we obtain $\mu_{X_n} (B_r (x_n)) \geq 1/2^n$. 
		As $n \geq 10$, we find $1/ 2^n - 2/3^{n -1} > 0$. 
		These imply
		\[
			\frac{\log \mu_{X_n} (B_r (x_n))}{\log r} \leq \frac{\log \frac{1}{2^n}}{\log \frac{2}{3^{n-1}}} < 1. 
		\]
		\item If $k = 1, 2, \ldots, n-8$ and  $2 / 3^{n - k} \leq r < 2 / 3^{n - (k + 1)}$, then we see that there exists $(i_1, \ldots, i_{n - k}) \in I_{n- k}$ such that
		\[
			x_n \in (F_{i_1} \circ \cdots \circ F_{i_{n - k}}) ([0, 1]) \subset B_r (x_n). 
		\]
		Therefore we have $\mu_{X_n} (B_r (x_n)) \geq 2^k / 2^n$. 
		Since $n \geq 10$, we have $2^k / 2^n - 2 / 3^{n - (k + 1)} > 0$. 
		Hence, we get
		\[
			\frac{\log \mu_{X_n} (B_r (x_n))}{\log r} \leq \frac{\log \frac{2^k}{2^n}}{\log \frac{2}{3^{n-(k +  1)}}} < 1. 
		\]
	\end{itemize}
	Therefore, we conclude that $\{ (X_n, \mu_{X_n}) \}_{n \in \mathbb{N}, n \geq 10}$ satisfies the condition U($2/3^7$). 
\end{ex}

Under the condition U($D$), the following example implies that 
$\udimH$, $\ldimH$, $\essudimH$, and $\essldimH$ are not upper semicontinuous with respect to the measured Gromov--Hausdorff distance and the box distance. 

\begin{ex}
	For $r > 0$, we set $S^1 (r) \coloneqq \{ (x, y) \in \mathbb{R}^2 \mid x^2 + y^2 = r^2 \}$. 
	$S^1 (r)$ has the distance induced by the Riemannian metric and the normalized Riemannian volume measure, so $S^1 (r)$ is an mm-space.  
	$\{ S^1 (r) \}_{0 < r < \pi^{-1}}$ satisfies the condition U(1). 
	We now prove it. 
	Fix $0 < r < \pi^{-1}$ and $x_r \in S^1 (r)$. 
	We show 
	\[
		\lim_{\theta \to 0} \frac{\log \mu_{S^1 (r)} (B_\theta (x_r))}{ \log \theta} = \sup_{0 < \theta < 1} \frac{\log \mu_{S^1 (r)} (B_\theta (x_r)) } {\log \theta} = 1.
	\] 
	If $0 < \theta \leq \pi r$, the equality $\mu_{S^1 (r)} (B_\theta (x_r)) = \theta / \pi r$ holds. 
	Thus, we have
	\[
		\frac{\log \mu_{S^1 (r)} (B_\theta (x_r))}{\log \theta} = 1 + \frac{\log \frac{1}{\pi r}}{\log \theta}. 
	\]
	If $\pi r < \theta < 1$, we find $\log \mu_{S^1 (r)} (B_\theta (x_r)) / \log \theta = 0$ since the equality $\mu_{S^1 (r)} (B_\theta (x_r)) = 1$ holds . 
	Therefore, we see that $\{ S^1 (r) \}_{0 < r < \pi^{-1}}$ satisfies the condition U(1). 
	
	$\{ S^1 (r)\}_{0 < r < \pi^{-1}}$ measured Gromov--Hausdorff converges to the one-point mm-space $* \coloneqq (\{ \ast \}, d_*, \delta_*)$ as $r \to 0$, where $\delta_*$ is the Dirac measure and $d_*$ is the trivial metric. 
	For $0 < r < \pi^{-1}$, we have $\udimH (\mu_{S_1 (r)}) = \ldimH (\mu_{S_1 (r)}) = \essudimH (\mu_{S_1 (r)}) = \essldimH (\mu_{S_1 (r)}) = 1$ and 
	all four dimensions of $\delta_*$ are equal to $0$. 
\end{ex}

\begin{lem}\label{lem: estimate ball}
	Let $\mathcal{L}^2$ be the two-dimensional Lebesgue measure. 
	For $R > 0$, $0 < r \leq 2R$, and $x \in B_R (0) \subset \mathbb{R}^2$, we have
	\[
		\frac{\pi}{4}r^2 \leq \mathcal{L}^2 (B_R (0) \cap B_r (x)) \leq \pi r^2. 
	\]
\end{lem}

\begin{proof}
	It is trivial that the latter inequality holds. 
	If $\|x\|_2 \leq r/ 2$, the first inequality clearly holds because $B_{r/2} (0) \subset B_R (0) \cap B_r (x)$. 
	If $\| x \|_2 > r / 2$, we set 
	\[
		z \coloneqq \left( 1 - \frac{r}{2 \| x \|_2} \right) x. 
	\]
	Fix $y \in B_{r/2} (z)$. Then we have
	\begin{align*}
		& \| y \|_2 \leq \| z \|_2 + \frac{r}{2} \leq R - \frac{r}{2} + \frac{r}{2} = R, \\
		& \| x - y \|_2 \leq \|x - z \|_2 + \| y - z \|_2 \leq \frac{r}{2} + \frac{r}{2} = r, 
	\end{align*}
	where $\| \cdot \|_2$ is the Euclidean norm. 
	Therefore we obtain $B_{r/2} (z) \subset B_R (0) \cap B_r (x)$. 
	This completes the proof. 
\end{proof}

Let $X$ be an mm-space. For $x \in X$, if the limit 
\[
	\lim_{r \to 0} \frac{\log \mu_X (B_r (x))}{\log r}
\]
exists, we write $d_{\mu_X} (x)$ for it. 

\begin{ex} \label{ex: counterexample essldimH}
	The Borel probability measure $\mu$ on $B_{1/10} (0) \subset (\mathbb{R}^2, \| \cdot \|_2) $ is defined by
	\[
		d \mu (x) \coloneqq \left( \frac{8}{\pi} + \frac{23}{5 \pi \|x\|_2} \right) \; d \mathcal{L}^2 (x).
	\]
	We set $B \coloneqq (B_{1/10} (0), \mu)$. 
	
	Let $J_n \coloneqq ([0, 1/ (100 (n+1))], | \cdot |, 100(n+1) \mathcal{L}^1|_{[0, 1/ (100 (n+1))]})$. 
	We identify $0 \in J_n$ with the origin of $B$. 
	Then we obtain the wedge sum $X_n \coloneqq B \vee J_n$. 
	We define the metric $d_{X_n}$ on $X_n$ as follows: for $x,y \in B$, $s, t \in J_n$
	\[
		d_{X_n} (x, t) = d_{X_n} (t, x) \coloneqq \|x\|_2 + t ,
	\] 
	 $d_{X_n} (x,y) = d_{X_n} (y,x) \coloneqq \| x - y \|_2$, and $d_{X_n} (s, t) = d_{X_n} (t, s) = |  s- t |$. 
	 We set
	 \[
	 	\mu_{X_n} \coloneqq \left( 1 - \frac{1}{100 (n + 1)} \right) \mu + \mathcal{L}^1|_{J_n}. 
	 \]
	 
	 We show that $(X_n, d_{X_n}, \mu_{X_n})$ satisfies the condition U($1/20$). 
	  It is easy to check that $d_{\mu_{X_n}} (t) = 1$ for $t \in J_n \setminus \{ 0 \}$. 
	 For $0 < r < 1/(100 (n+1))$, we have
	 \begin{align*}
	 	\mu_{X_n} (B_r (0)) &= \left( 1 - \frac{1}{100 (n+1)} \right)\mu (B_r (0)) + r \\
		&
		= \left( 1 - \frac{1}{100 (n+1)} \right) \left( 8 r^2 + \frac{46}{5} r \right) +r
	 \end{align*}
	 Therefore $d_{\mu_{X_n}} (0) = 1$. 
	For $x \in B \setminus \{0 \}$,  the probability density function of $\mu$ is bounded above and below by positive constants on a sufficiently small closed ball centered at $x$. 
	By Lemma \ref{lem: estimate ball}, we have $d_{\mu_{X_n}} (x) = 2$. 
	
	Next, we show that $\mu_{X_n} (B_r (x)) \geq r^{d_{\mu_{X_n}(x)}}$ for $x \in X_n$ and $0 < r < 1/ 20$. 
	\begin{itemize}
		\item If $x \in J_n \setminus \{ 0 \}$ and $0 < r < 1/ (100 (n+1))$,  we have
		\begin{align*}
			\mu_{X_n} (B_r (x)) &\geq \mathcal{L}^1|_{J_n} (B_r (x)) \\
			&= \min \{ x, r \} + \min \left \{ r, \frac{1}{100 (n+1)} - x \right \} \geq r.
		\end{align*}
		If $1/ (100 (n+1)) \leq r < 1 /20$,  we have
		\begin{align*}
			&\mu_{X_n} (B_r (x)) \\
			&\geq \frac{1}{100 (n+1)} + \left( 1 - \frac{1}{100 (n+1)} \right) \mu (B_{r - 1/(100 (n+1))} (0)) \\
			& \geq \frac{1}{100 (n+1)} + \left( 1 - \frac{1}{100 (n+1)} \right) \frac{46}{5} \left (r - \frac{1}{100 (n+1)} \right) \\
			& \geq \frac{1}{100 (n+1)} + \left (r - \frac{1}{100 (n+1)} \right) = r = r^{d_{\mu_{X_n}(x)}}. 
		\end{align*}
		\item If $x \in B \setminus \{ 0 \}$, by Lemma \ref{lem: estimate ball}, we obtain 
		\[
			\mu_{X_n} (B_r (x)) \geq \left( 1 - \frac{1}{100 (n+1)} \right) \frac{8}{\pi} \frac{\pi}{4} r^2 \geq r^2 = r^{d_{\mu_{X_n}(x)}}. 
		\]
		\item If $x = 0$, we have
		\begin{align*}
			\mu_{X_n (B_r (0))} & \geq \left( 1 - \frac{1}{100 (n+1)} \right) \mu (B_r (0)) \\
			& \geq \left( 1 - \frac{1}{100 (n+1)} \right) \frac{46}{5} r \geq r = r^{d_{\mu_{X_n} (0)}}. 
		\end{align*}
	\end{itemize}
	Thus we conclude that $X_n$ satisfies the condition U($1/20$). 
	 
	 We see that $\{ X_n\}_{n \in \mathbb{N}}$ measured Gromov--Hausdorff converges to $B$. 
	 By the same argument as above, we have $d_\mu (0) = 1$ and $d_\mu (x) = 2$ for $x \in B \setminus \{ 0 \}$. 
	 Therefore we obtain 
	 \[
	 	2 = \essldimH (\mu) > \liminf_{n \to \infty} \essldimH (\mu_{X_n}) = 1. 
	 \]
\end{ex}

Let $p, q > 0$ with $p + q = 1$, and let $X$ and $Y$ be mm-spaces such that $\diam (X)$ and $\diam (Y)$ is finite. 
For $r > \max \{\diam (X), \diam (Y) \}$, we set
\[
	X^p +_r Y^q \coloneqq (X \sqcup Y, d_{X^p+_rY^q}, p \mu_X + q \mu_Y), 
\] 
where, for $z, w \in X \sqcup Y$, we define $d_{X^p+_rY^q} (z, w)$ by 
\[
	d_{X^p+_rY^q} (z, w) \coloneqq 
		\begin{cases*}
			d_X (z, w) & if $z, w \in X$ \\
			r & if $z \in X, w \in Y$ or $w \in X, z \in Y$\\
			d_Y (z, w) & if $z, w \in Y$
		\end{cases*}
		. 
\]
We see that $X^p +_r Y^q$ is an mm-space. 

\begin{ex} \label{ex: not closed under UD}
	For $0 < r < (2\pi)^{-1}$, we set $X_r \coloneqq S^1 (r)^{1/2} +_1 S^1 (r)^{1/2}$. 
	Fix $x \in X_r$. 
	If $0 < \theta \leq \pi r$, we have $\mu_{X_r} (B_\theta (x)) = \theta /2 \pi r$. 
	Thus, we obtain 
	\[
		\frac{\log \mu_{X_r} (B_\theta (x))}{\log \theta} = 1 + \frac{\log \frac{1}{2 \pi r}}{\log \theta}. 
	\]
	If $\pi r < \theta < 2^{-1}$, the equality $\mu_{X_r} (B_\theta (x)) = 1/ 2$ implies $\log 2^{-1}/ {\log \theta} \leq 1$. 
	Hence, we find
	\[
		\lim_{\theta \to 0} \frac{\log \mu_{X_r} (B_\theta (x))}{\log \theta} = \sup_{0 < \theta < 1/2} \frac{\log \mu_{X_r} (B_\theta (x))}{\log \theta} = 1. 
	\]
	Therefore, we see that $X_r$ satisfies the condition U($1/2$). 
	
	$\{ X_r \}_{0 < r < (2 \pi)^{-1}}$ measured Gromov--Hausdorff converges to $X \coloneqq *^{1/2} +_1 *^{1/2} $. 
	In particular, Proposition \ref{prop: mgh and box} implies that it box converges to $X$. 
	However, $X$ does not satisfy the condition U($1/2$). 
	Indeed, for $x \in X$ and $0 < \theta < 1/2$, we have $\mu_X (B_\theta (x)) = 1/ 2$. 
	Therefore, we obtain
	\[
		0 = \lim_{\theta \to 0} \frac{\log \mu_X (B_\theta (x))}{\log \theta} = \inf_{0 < \theta < 1/ 2}\frac{\log \mu_X (B_\theta (x))}{\log \theta}
		\neq \sup_{0 < \theta < 1/ 2}\frac{\log \mu_X (B_\theta (x))}{\log \theta}
	\]
	This implies that $\mathcal{X}_{\mathrm{U}(D)} \subset \mathcal{X}$ is not necessarily a closed set with respect to the box topology. 
\end{ex}

\subsection{Condition L($D$)}

\begin{ex} \label{ex: not closed under LD}
	For $n \in \mathbb{N}$, we set $X_n \coloneqq \left(\left \{ i/ n \right\}_{i = 1}^n, | \cdot |, n^{-1} \sum_{i = 1}^n \delta_{i/n} \right)$. 
	 $X_n$ is an mm-space. 
	For $i = 1, \ldots, n$, it is obvious that 
	\[
		\lim_{r \to 0} \frac{\log \mu_{X_n} (B_r (i/n))}{\log r} = \inf_{0 < r < 1/2} \frac{\log \mu_{X_n} (B_r (i/n))}{\log r} = 0. 
	\]
	Thus, we see that $X_n$ satisfies the condition L(1/2). 
	Moreover, $\udimH (\mu_{X_n}) = \ldimH (\mu_{X_n}) = \essudimH (\mu_{X_n}) = \essldimH (\mu_{X_n}) = 0$ holds. 	
	
	$\{ X_n \}_{n \in \mathbb{N}}$ measured Gromov--Hausdorff converges to $I \coloneqq ([0, 1], |\cdot|, \mathcal{L}^1|_{[0, 1]})$. 
	We also have $\udimH (\mathcal{L}^1|_{[0, 1]}) = \ldimH (\mathcal{L}^1|_{[0, 1]}) = \essudimH (\mathcal{L}^1|_{[0, 1]}) = \essldimH (\mathcal{L}^1|_{[0, 1]}) = 1$. 
	Therefore, this is an example of Theorems \ref{thm: uppersemi, box} and \ref{thm: uppersemi, mGH}. 
	This example shows that none of the four dimensions is lower semicontinuous with respect to either the measured Gromov--Hausdorff distance or  the box distance. 
	
	$I$ does not satisfy the condition L(1/2). 
	Indeed, let $x \coloneqq 1/ 2$. 
	For $0 < r < 1/2$, we have $\mathcal{L}^1|_{[0, 1]} (B_r (x)) = 2r$. 
	This implies 
	\begin{align*}
		1 &= \lim_{r \to 0} \frac{\log \mathcal{L}^1|_{[0, 1]} (B_r (x))}{\log r} \\
		&= \sup_{0 < r < 1/2}\frac{\log \mathcal{L}^1|_{[0, 1]} (B_r (x))}{\log r} 
		\neq \inf_{0 < r < 1/2}\frac{\log \mathcal{L}^1|_{[0, 1]} (B_r (x))}{\log r}. 
	\end{align*}
	Thus, $\mathcal{X}_{\mathrm{L}(D)} \subset \mathcal{X}$ is not necessarily a closed set with respect to the box topology. 
\end{ex}

\begin{ex} \label{ex: direct sum of S^1 and * -1}
	For $n \geq 2$, we set $X_n \coloneqq (S^1(1))^{1-n^{-1}} +_{2 \pi} *^{n^{-1}}$. 
	Fix $x_n \in X_n$. 
	If $x_n \in S^1 (1)$, we have 
	\begin{align*}
		\lim_{r \to 0} \frac{\log \mu_{X_n} (B_r (x_n))}{ \log r }
		&=
		\lim_{r \to 0} \left(1 + \frac{\log \pi^{-1}}{\log r} + \frac{\log (1 - n^{-1})}{\log r} \right) \\
		&=
		\inf_{0 < r < 1} \left( 1 + \frac{\log \pi^{-1}}{\log r} + \frac{\log (1 - n^{-1})}{\log r} \right)= 1. 
	\end{align*}
	If $x_n = *$, we obtain 
	\[
		\lim_{r \to 0} \frac{\log \mu_{X_n} (B_r (*))} { \log r} = \inf_{0 < r < 1} \frac{\log n^{-1}} { \log r} = 0.
	\] 
	Therefore, $\{ X_n \}_{n \geq 2}$ satisfies the condition L(1). 
	Moreover, the equality 
	$\ldimH (\mu_{X_n}) = \essldimH (\mu_{X_n})= 0$ holds. 
	
	Letting $n \to \infty$, we see that $\{ X_n \}_{n \geq 2}$ box converges to $S^1 (1)$.  
	We have $\ldimH (\mu_{S^1 (1)}) = \essldimH (\mu_{S^1 (1)})= 1$. 
	Hence, we obtain 
	\begin{align*} 
		0 = \liminf_{n \to \infty} \ldimH (\mu_{X_n}) &< \ldimH (\mu_{S^1 (1)}) = 1, \\
		0 = \liminf_{n \to \infty} \essldimH (\mu_{X_n})&< \essldimH (\mu_{S^1 (1)})= 1. 
	\end{align*}
	Under the condition L($D$), these imply that $\ldimH$ and $\essldimH$ are not necessarily lower semicontinuous with respect to the box topology. 
\end{ex}

\begin{ex} \label{ex: counterexample, essudimH}
	We set $L \coloneqq ([0, 16], | \cdot |, \nu)$, where
	\[
		d \nu (t) \coloneqq \frac{2t}{16^2} \; dt. 
	\]
	Let $B_{1/  (n+2)} (0) \subset \mathbb{R}^2$ and let $B_n \coloneqq (B_{1/  (n+2)} (0), \| \cdot \|_2, \mathcal{L}^2|_{B_{1/  (n+2)} (0)}) $. 
	In the same way as Example \ref{ex: counterexample essldimH}, we identify $0 \in L$ with the origin of $B_n$ and obtain the metric space $X_n \coloneqq (L \vee B_n, d_{X_n})$, where the the restrictions of $d_{X_n}$ to $L$ and $B_n$ agree with $| \cdot |$ and $\| \cdot \|_2$, respectively, and $d_{X_n} (x, t) = d_{X_n} (t, x) = t + \| x \|_2$ for $x \in B_n$ and $t \in L$.
	We define the Borel probability measure $\mu_{X_n}$ on $X_n$ as
	\[
		\mu_{X_n} \coloneqq \left( 1 - \frac{1}{4 (n+2)^2} \right) \nu + \frac{1}{4 \pi} \mathcal{L}^2|_{B_{1/  (n+2)} (0)}. 
	\]
	
	We show that $X_n$ satisfies the condition L($1$). 
	By Lemma \ref{lem: estimate ball}, we have $d_{\mu_{X_n}} (x) = 2$ for $x \in B_n \setminus \{ 0 \}$. 
	Moreover, we see that $d_{\mu_{X_n}} (t) = 1$ for $t \in L \setminus \{ 0 \}$. 
	For $0 < r < (n+2)^{-1}$, we have
	\[
		\mu_{X_n} (B_r (0)) = \left(1 - \frac{1}{4 (n+2)^2} \right)\frac{1}{16^2} r^2 + \frac{1}{4 } r^2. 
	\]
	Thus we see that $d_{\mu_{X_n}} (0) = 2$. 
	To complete the claim, it suffices to show that $\mu_{X_n} (B_r (x)) \leq r^{d_{\mu_{X_n}} (x)}$ for $0 < r < 1$ and $x \in X_n$. 
	Fix $0 < r < 1$. 
	If $x \in B_n$, we set 
	$s \coloneqq \max \{ r - \| x \|_2, 0 \}$. 
	Then we have
	\begin{align*}
		\mu_{X_n} (B_r (x)) &\leq \frac{1}{4}r^2 + \left( 1 - \frac{1}{4 (n+2)^2} \right) \int_0^s \frac{2t}{16^2} dt
		\\
		&\leq \frac{1}{4}r^2 + \frac{r^2}{16^2} \leq r^2 = r^{d_{\mu_{X_n}}(x)}. 
	\end{align*}
	If $t \in L \setminus \{ 0 \}$, we obtain
	\begin{align*}
		\mu_{X_n} (B_r (t)) &\leq \int_{\max \{t-r, 0\} }^{\min \{t+r, 16 \}} \frac{2 s}{16^2} ds + \frac{1}{4} r^2
		 \leq
		\int_{\max \{t-r, 0\} }^{\min \{t+r, 16 \}} \frac{2}{16} ds + \frac{1}{4}r^2 \\
		& \leq \frac{1}{4} r + \frac{1}{4}r^2 \leq r = r^{d_{\mu_{X_n}}(t)}. 
	\end{align*}

	$\{ X_n \}_{n \in \mathbb{N}}$ measured Gromov--Hausdorff converges to $L$. 
	It is easy to check that $d_\nu (0) = 2$ and $d_\nu (t) = 1$ for $t \in L \setminus \{ 0 \}$. 
	Therefore
	\[
		2 = \limsup_{n \to \infty} \essudimH (\mu_{X_n}) > \essudimH (\nu) = 1. 
	\]
\end{ex}

\begin{ex}
	For $n \geq 2$, we set $X_n \coloneqq (S^1(1))^{n^{-1}} +_{2 \pi} *^{1-n^{-1}}$. 
	In the same way as in Example \ref{ex: direct sum of S^1 and * -1}, we see that $\{ X_n \}_{n \geq 2}$ satisfies the condition L(1) and obtain that
	$\udimH (\mu_{X_n}) = \essudimH (\mu_{X_n}) = 1$ and $\udimH (\delta_*) = \essudimH (\delta_*) = 0$. 
	As $n \to \infty$, we find that $\{ X_n \}_{n \geq 2}$ box converges to $*$. 
	Therefore, under the condition L($D$),  $\udimH$ and $\essudimH$ are not necessarily upper semicontinuous with respect to the box distance. 
\end{ex}

\end{document}